\documentclass{amsart}
\usepackage{graphicx}
\usepackage{amsmath,amssymb}
\newtheorem{thm}{Theorem}[section]
\newtheorem{cor}[thm]{Corollary}
\newtheorem{lem}[thm]{Lemma}
\newtheorem{prop}[thm]{Proposition}
\theoremstyle{definition}

\theoremstyle{remark}

\numberwithin{equation}{section}

\begin{document}

\title[On the integral  functional equations]{On the integral d'Alembert's and
 Wilson's functional equations }
\author[B. Bouikhalene and  E. Elqorachi]{Bouikhalene Belaid and   Elqorachi Elhoucien}


\begin{abstract} Let $G$ be a locally compact group, and let $K$ be a compact subgroup of $G$. Let $\mu : G\longrightarrow\mathbb{C}\backslash\{0\}$ be a  character of  $G$. In this paper, we deal with the integral equations
 $$W_{\mu}(K):\; \;\int_{K}f(xkyk^{-1})dk+\mu(y)\int_{K}f(xky^{-1}k^{-1})dk=2f(x)g(y),$$
 and
 $$D_{\mu}(K):\; \;\int_{K}f(xkyk^{-1})dk+\mu(y)\int_{K}f(xky^{-1}k^{-1})dk=2f(x)f(y)$$
 for all $x, y\in G$ where $f, g: G\longrightarrow \mathbb{C}$, to be determined, are complex continuous functions on $G$.
 When $K\subset Z(G)$, the center of $G$, $D_{\mu}(K)$ reduces to the new version of d'Almbert's functional equation
 $f(xy)+\mu(y)f(xy^{-1})=2f(x)f(y)$, recently studied by Davison [18] and Stetk{\ae}r [35].
  We derive the following link between the solutions of $W_{\mu}(K)$ and $D_{\mu}(K)$ in the following way :
  If $(f,g)$ is a solution of equation $W_{\mu}(K)$ such  that $C_{K}f=\int_{K}f(kxk^{-1})d\omega_{K}(k)\neq 0$ then $g$ is a solution of $D_{\mu}(K)$. This result is used to establish the superstability problem of $W_{\mu}(K)$.  In the case where $(G,K)$ is a central pair, we show that the solutions are expressed by means of $K$-spherical functions and related functions. Also we give  explicit formulas  of solutions of $D_{\mu}(K)$ in terms of irreducible representations of $G$. These formulas generalize Euler's formula $\cos(x)=\frac{e^{ix}+e^{-ix}}{2}$ on $G=\mathbb{R}$.
\end{abstract}
\maketitle
\section{Introduction and Preliminaries}
\subsection{}
Throughout this paper, $G$ will be a locally compact group, $K$ be a
compact subgroup of $G$ and $dk$ the normalized Haar measure of the
compact group $K$. The unit element of $G$ is denoted by $e$. The
center of $G$ is dented by $Z(G)$. For any function $f$ on $G$
we define the function $\check{f}(x)=f(x^{-1})$ for any $x\in G$. The space of all complex continuous functions on $G$ having compact support is designed by $\mathcal{K}(G)$.  We denote by $\mathcal{C}(G)$ the space of all complex continuous functions on $G$. For each fixed $x\in G$, we define the left translation operator by $(L_{x}f)(y)=f(x^{-1}y)$ for all $y\in G$.\\
For a given character  $\mu :G \longrightarrow
\mathbb{C}\backslash\{0\}$ we consider the following integral
equation
\begin{equation}\label{eq11}
 \int_{K}f(xkyk^{-1})dk+\mu(y)\int_{K}f(xky^{-1}k^{-1})dk=2f(x)g(y),\;\;x,y\in G.
\end{equation}
This equation is a generalization of the following  functional
equations :
\begin{equation}\label{eq12}
 \int_{K}f(xkyk^{-1})dk+\mu(y)\int_{K}f(xky^{-1}k^{-1})dk=2f(x)f(y),\;\;x,y\in G,
\end{equation}
 which was studied in [7] when $\mu=1$ and $(G,K)$ is a central pair. \\
If $K\subset Z(G)$ the subgroup center of $G$ and $f=g$,  (1.1)
becomes d'Alembert's functional equation
\begin{equation}\label{eq14}
 f(xy)+\mu(y)f(xy^{-1})=2f(x)f(y),\;\;x,y\in G.
\end{equation}
In the case where $\mu=1$, many authors studied the functional equation (1.3) (see [3], [15], [16], [17], [26], [29], [30], [31], [32], [33], [34], [35], [36], [37], [38], [49]).\\
When $f(kxh)=f(x)$ for any $x\in G$ and $k, h\in K$, we obtain the
functional equation
\begin{equation}\label{eq12}
 \int_{K}f(xky)dk+\int_{K}f(xky^{-1})dk=2f(x)g(y),\;\;x,y\in G.
\end{equation}
 If $K\subset Z(G)$, (1.1) reduces to the following version of Wilson's functional equation
\begin{equation}\label{eq13}
 f(xy)+\mu(y)f(xy^{-1})=2f(x)g(y),\;\;x,y\in G.
\end{equation}
IF $K\subset Z(G)$ and $\mu=1$, (1.1) becomes the Wilson's
functional equation
\begin{equation}\label{eq14}
 f(xy)+f(xy^{-1})=2f(x)g(y),\;\;x,y\in G.
\end{equation}
If $f(xk)=\overline{\chi(k)}f(x)$, where $x\in G$, $k\in K$ and
$\chi$ is a unitary character of $K$ we obtain the functional
equation
\begin{equation}
 \int_{K}f(xky)\overline{\chi(k)}k+\mu(y)\int_{K}f(xky^{-1})\overline{\chi(k)}dk=2f(x)g(y),\;\;x,y\in G.
\end{equation}
If $G$ is compact we can take $K=G$ and consider the functional
equation
\begin{equation}
 \int_{G}f(xtyt^{-1})dt+\mu(y)\int_{G}f(xty^{-1}t^{-1})dt=2f(x)f(y),\;\;x,y\in G.
\end{equation}
The equations (1.4), (1.7) and (1.8) were  studied  in [1], [7],
[9], [10] and [21].
 The  functional equation (1.6) appeared in several works by
H. Stetk\ae r, see for example [31], [32] and [33]. For equation
(1.3), we refer  to the recent studies by Davison [18] and  Stetk\ae
r [36].
 \subsection{Recall on the central pairs }
For a function $f$ on $G$, we say that the function $f$ is
$K$-central if $f(kx)=f(xk)$ for all $k\in K$ and for all $x\in G$.
We put $\mathcal{K}_{K}(G)=\{ f\in \mathcal{K}(G) : f(kx)=f(xk),
x\in G, k\in K\}$. Under convolution, denoted $\star$,
$\mathcal{K}_{K}(G)$ is a subalgebra of the algebra
$\mathcal{K}(G)$. We recall (see [7]) that the pair $(G,K)$ is said
to be a central pair if the algebra
($\mathcal{K}_{K}(G),\star)$ is commutative.\\
A non-zero continuous function $\varphi$ on $G$ is called
$K$-spherical function, if
\begin{equation}
 \int_{K}\varphi(xkyk^{-1})dk=\varphi(x)\varphi(y),  x, y\in G.
\end{equation}
for all $x, y\in G$. We will say that a function $f\in
\mathcal{C}(G)$ satisfying
\begin{equation}
 \int_{K}f(xkyk^{-1})dk=f(x)\varphi(y)+f(y)\varphi(x), \;\; x, y\in G.
\end{equation}
is associated with the $K$-spherical function $\varphi$. Let
$C_{K}:\mathcal{C}(G)\longrightarrow \mathcal{C}(G)$ be the operator
given by $$(C_{K}f)(x)=\int_{K}f(kxk^{-1})dk, \; x\in G.$$ By easy
computations we show that $f$  is $K$-central if and only if
$C_{K}f=f$. For more results on the operator $C_{K}$ we refer to [7,
Propositions 2.1, Proposition 2.2]. We say that $f\in
\mathcal{C}(G)$ satisfies the Kannappan type condition if
$$\int_{K}\int_{K}f(zkxk^{-1}hyh^{-1})dkdh=\int_{K}\int_{K}f(zkyk^{-1}hxh^{-1})dkdh, \; x, y\in G \;\;\; \; (*)$$
When $K\subset Z(G)$, $(*)$ reduces to Kannappan condition $f(xyz)=f(yxz)$ for all $x,y, z\in G$ (see [27]).\\
The results of the present paper are organised as follows :  In
section 2 we establish relationship between functional equation
$(1.1)$ and $(1.2)$. In Theorem 2.3 we show that if $(f,g)$ is a
solution of (1.1) such that $f\neq 0$ and $C_{K}f\neq 0$, without
the assumption that $f$ satisfies (*), then $g$ is a solution of
(1.2). In section 4 we show that if $(G,K)$ is a central pair and
$f$ is a solution of $(1.2)$, then $f$  has the form
$f=\frac{\varphi+\mu\check{\varphi}}{2}$ where $\varphi$ is a
$K$-spherical function.
 Furthermore we give a complete description of the solutions of equations  of (1.1) and (1.2) in the case where $(G,K)$ is a central pair.  The solutions are expressed by means of $K$-spherical functions and solutions of the functional equation
\begin{equation}\int_{K}f(xkyk^{-1})dk=f(x)\varphi(y)+f(y)\varphi(x),\;\;x,y\in G \end{equation}
in which $\varphi$ is a $K$-spherical function. In Corollaries  4.3
and 4.4 we give  explicit formulas  of solutions of (1.2) and (1.8)
in terms of irreducible representations of $G$. Theses  formulas
generalize Euler's formula $\cos(x)=\frac{e^{ix}+e^{-ix}}{2}$ on
$G=\mathbb{R}$. In the last section we study stability [48] and
Baker's superstability (see [5] and [6]) of the functional equations
(1.1), (1.2), (1.3), (1.4), (1.5), (1.6) and (1.7). For more
information concerning the stability problem we refer to [3], [5],
[6], [11], [12], [22], [40],[41], [42], [43], [44], [45], [46], [47]
and [48]. The results of the last sections generalize the ones
obtained in [12] and [21].
\section{ General properties  of equations  $W_{\mu}(K)$ }
In this section we deal with the integral Wilson's  functional
equation (1.1) on a locally compact group $G$. We prove, without the
assumption that $f$ satisfies (*), that if $(f,g)$ is a solution of
Wilson's functional equation (1.1) then $g$ is a solution of
d'Alembert's functional equation (1.2).\\\\ For later use we need
the following proposition
\begin{prop}
Let $G$ be a locally compact group. Let $\mu: G \longrightarrow \mathbb{C}\backslash\{0\}$ be a continuous character of $G$ and let $\varphi\in \mathcal{C}(G)$ be a $K$-spherical function.  Then\\
i) $\mu \varphi$ is a $K$-spherical function.\\
ii) $\frac{\varphi+\mu\check{\varphi}}{2}$ is  a solution of (1.2).\\
iii) Assuming  $(f,g)$ is a solution of (1.1) we have : (1) the pair
$(L_{x}f,g)$ for all $x\in G$ is a solution of (1.1), and (2) the
pair $(C_{K}f,g)$ is a solution of (1.1).
\end{prop}
\begin{proof}
We get i) and ii) by easy  computations\\
iii) Let $x\in G$. For all $y, z\in G$ we have
$$\int_{K}C_{K}(L_{x^{-1}}f)(ykzk^{-1})dk+\mu(z)\int_{K}C_{K}(L_{x^{-1}}f)(ykz^{-1}k^{-1})dk$$
$$=\int_{K}\int_{K}(L_{x^{-1}}f)(hykzk^{-1}h^{-1})dkdh+
\mu(z)\int_{K}\int_{K}(L_{x^{-1}}f)(hykz^{-1}k^{-1}h^{-1})dkdh$$
$$=\int_{K}\int_{K}f(xhykzk^{-1}h^{-1})dkdh+
\mu(z)\int_{K}\int_{K}f(xhykz^{-1}k^{-1}h^{-1})dkdh$$
$$=\int_{K}\int_{K}f(xhk^{-1}ykzh^{-1})dkdh+
\mu(z)\int_{K}\int_{K}f(xhk^{-1}ykz^{-1}h^{-1})dkdh$$
$$=\int_{K}\int_{K}f(xhk^{-1}ykzh^{-1})dkdh+
\mu(z)\int_{K}\int_{K}f(xhk^{-1}ykz^{-1}h^{-1})dkdh$$
$$=\int_{K}\int_{K}f(xk^{-1}ykhzh^{-1})dkdh+
\mu(z)\int_{K}\int_{K}f(xk^{-1}ykhz^{-1}h^{-1})dkdh$$
$$=\int_{K}\int_{K}f(xkyk^{-1}hzh^{-1})dkdh+
\mu(z)\int_{K}\int_{K}f(xkyk^{-1}hz^{-1}h^{-1})dkdh$$
$$=2\int_{K}f(xkyk^{-1})d\omega_{K}(k)g(z)$$
$$=2\int_{K}(L_{x^{-1}}f)(kyk^{-1})dkg(z)$$
$$=2C_{K}(L_{x^{-1}}f)(y)g(z).$$
\end{proof}
\begin{prop}
Let $G$ be a locally compact group. Let $\mu: G \longrightarrow \mathbb{C}\backslash\{0\}$ be a character of $G$ and let $f, g\in \mathcal{C}(G)$ such that $f\neq 0$ be a solution of (1.1). Then\\
i) $g$ is $K$-central.\\
ii)  $g(x)=\mu(x)g(x^{-1})$ for all $x\in G$.\\
iii) If $f$ is $K$-central with $f(e)=0$, then $f(x)=-\mu(x)f(x^{-1})$ for all $x\in G$.\\
iv) $g(e)=1$.\\
\end{prop}
\begin{proof}
i) by easy computations.\\ ii) let $a\in G$ such that $f(a)\neq 0$,
then for any $y\in G$ we have
\begin{equation}
\int_{K}f(aky^{-1}k^{-1})dk+\mu(y^{-1})\int_{K}f(akyk^{-1})dk=2f(a)g(y^{-1}).
\end{equation}
Multiplying (2.1) by $\mu(y)$ and using the fact that
$\mu(yy^{-1})=1$ we get for all $y\in G$
\begin{eqnarray*}
2f(a)\mu(y)g(y^{-1})&=&\int_{K}f(akyk^{-1})k+\mu(y)\int_{K}f(aky^{-1}k^{-1})dk
\\&=&2f(a)g(y)
 \end{eqnarray*}
which implies that $g(y)=\mu(y)g(y^{-1})$ for any $y\in G$.\\ iii)
since $f$ is a solution of (1.1) we get by setting $x=e$
 in (1.1)  that
$$\int_{K}f(kyk^{-1})dk+\mu(y)\int_{K}f(ky^{-1}k^{-1})dk=2f(e)g(y).$$
Since $f$ is $K$-central and $f(e)=0$ we get for all $y\in G$ that
$f(y)+\mu(y)f(y^{-1})=0$.  By easy computations we get the
remainder.
\end{proof}
 The next theorem is the main result of this section. We establish a relation between Wilson's functional equation (1.1) and d'Alembert's functional equation (1.2) on a locally compact group $G$, without the assumption that $f$ satisfies ($*$).
 \begin{thm}
Let $G$ be a locally compact group. Let $f, g\in \mathcal{C}(G) $ be
a solution of Wilson's functional equation (1.1) such that
$C_{K}f\neq 0$. Then $g$ is a solution of d'Alembert's functional
equation (1.2).
\end{thm}
\begin{proof}
By getting ideas from [13] and [36],  and [19]  we discuss  the following possibilities :\\
 The first possibility is  $f(x)=-\mu(x)f(x^{-1})$ for all $x\in G$. We let $x\in G$ and we put
$$\Phi_{x}(y)=\int_{K}g(xkyk^{-1})dk+\mu(y)\int_{K}g(y^{-1}kxk^{-1})dk-2g(x)g(y), \; y\in G.$$
According to Proposition 2.1 and the fact that $(f,g)$ is a solution
of (1.1) we get for any $x, y, z\in G$ that
$$2f(z)\Phi_{x}(y)+2f(y)\Phi_{x}(z)=
2f(z)[\int_{K}g(xkyk^{-1})dk$$$$+\mu(y)\int_{K}g(y^{-1}kxk^{-1})dk-2g(x)g(y)]+
2f(y)[\int_{K}g(xkzk^{-1})dk$$$$+\mu(z)\int_{K}g(z^{-1}kxk^{-1})dk-2g(x)g(z)]
=\int_{K}\int_{K}f(zhxkyk^{-1}h^{-1})dkdh$$$$+\mu(xy)\int_{K}\int_{K}f(zhy^{-1}kx^{-1}k^{-1}h^{-1})
dkdh$$$$+\mu(y)[\int_{K}\int_{K}f(zhy^{-1}kxk^{-1}h^{-1})dkdh
+\mu(y^{-1}x)\int_{K}\int_{K}f(zhx^{-1}kyk^{-1}h^{-1})dkdh]$$$$
\int_{K}\int_{K}f(yhxkzk^{-1}h^{-1})dkdh+\mu(xz)\int_{K}\int_{K}f(yhz^{-1}kx^{-1}k^{-1}h^{-1})
dkdh$$$$+\mu(z)[\int_{K}\int_{K}f(yhz^{-1}kxk^{-1}h^{-1})dkdh+$$$$
\mu(z^{-1}x)\int_{K}\int_{K}f(yhx^{-1}kzk^{-1}h^{-1})dkdh]$$$$-
\int_{K}\int_{K}f(zhxkyk^{-1}h^{-1})dkdh$$$$-
\mu(y)\int_{K}\int_{K}f(zhxky^{-1}k^{-1}h^{-1})dkdh-
\mu(x)[\int_{K}\int_{K}f(zhx^{-1}kyk^{-1}h^{-1})dkdh$$$$+
\mu(y)\int_{K}\int_{K}f(zhx^{-1}ky^{-1}k^{-1}h^{-1})dkdh]-
\int_{K}\int_{K}f(yhxkzk^{-1}h^{-1})dkdh$$$$-
\mu(z)\int_{K}\int_{K}f(yhxkz^{-1}k^{-1}h^{-1})dkdh-
\mu(x)[\int_{K}\int_{K}f(yhx^{-1}kzk^{-1}h^{-1})dkdh$$$$+
\mu(z)\int_{K}\int_{K}f(yhx^{-1}kz^{-1}k^{-1}h^{-1})dkdh]=
\mu(x)\mu(y)\int_{K}\int_{K}f(zhy^{-1}kx^{-1}k^{-1}h^{-1})dkdh$$$$+
\mu(y)\int_{K}\int_{K}f(zhy^{-1}kxk^{-1}h^{-1})
dkdh+\mu(xz)\int_{K}\int_{K}f(yhz^{-1}kx^{-1}k^{-1}h^{-1})
dkdh$$$$+\mu(z)\int_{K}\int_{K}f(yhz^{-1}kxk^{-1}h^{-1})dkdh-
\mu(y)\int_{K}\int_{K}f(zhxky^{-1}k^{-1}h^{-1})dkdh$$$$-\mu(x)\mu(y)
\int_{K}\int_{K}f(zhx^{-1}ky^{-1}k^{-1}h^{-1})dkdh-\mu(z)
\int_{K}\int_{K}f(yhxkz^{-1}k^{-1}h^{-1})dkdh$$$$-\mu(x)\mu(z)
\int_{K}\int_{K}f(yhx^{-1}kz^{-1}k^{-1}h^{-1})dkdh$$$$=
2\mu(y)\int_{K}f(zky^{-1}k^{-1})g(x)dk+2\mu(z)\int_{K}f(ykz^{-1}k^{-1})g(x)dk$$$$
-\mu(z)\int_{K}\int_{K}f(yhxkz^{-1}k^{-1}h^{-1})dkdh-\mu(x)\mu(y)
\int_{K}\int_{K}f(zhx^{-1}ky^{-1}k^{-1}h^{-1})dkdh$$$$
-\mu(y)\int_{K}\int_{K}f(zhxky^{-1}k^{-1}h^{-1})dkdh-\mu(x)\mu(z)
\int_{K}\int_{K}f(yhx^{-1}kz^{-1}k^{-1}h^{-1})dkdh$$$$=
2\mu(y)g(x)\int_{K}f(zky^{-1}k^{-1})dk-2\mu(z)g(x)\mu(yz^{-1})\int_{K}f(zky^{-1}k^{-1})dk$$$$
-\mu(x)\mu(y)\int_{K}\int_{K}f(zhx^{-1}ky^{-1}k^{-1}h^{-1})dkdh$$$$+
\mu(z)\mu(yxz^{-1})\int_{K}\int_{K}f(zhx^{-1}ky^{-1}k^{-1}h^{-1})dkdh
$$$$+\mu(y)\mu(zxy^{-1})\int_{K}\int_{K}f(yhx^{-1}kz^{-1}k^{-1}h^{-1})dkdh$$$$
-\mu(x)\mu(z)\int_{K}\int_{K}f(yhx^{-1}kz^{-1}k^{-1}h^{-1})dkdh=$$$$0.$$
Then for any $x\in   G$ we have \begin{equation}
f(z)\Phi_{x}(y)+f(y)\Phi_{x}(z)=0, \; x, y \in G.
\end{equation}
Since $f\neq 0$, then there exists $a\in G$ such that $f(a)=0$.
According to (2.2) we get that $f(a)\Phi_{x}(y)+f(y)\Phi_{x}(a)=0$
and then $\Phi_{x}(y)=-\frac{\Phi_{x}(a)}{f(a)}f(y)=c_{x}f(y)$ for
any $y\in G$. By setting $y=a$ in (2.2) we get that
$c_{x}f(a)^{2}=0$. This implies that $c_{x}=0$ and then $\Phi_{x}=0$
for any $x\in G$. According to Proposition 2.2 we get we get by
using the fact that $\Phi_{x}(y)=0$ for any $x ,y\in G$ that
\begin{eqnarray*}2g(x)g(y)&=&\int_{K}g(xkyk^{-1})dk+\mu(y)\int_{K}g(y^{-1}xkxk^{-1})dk\\
&=&\mu(x)\mu(y)\int_{K}g(y^{-1}kx^{-1}k^{-1})dk+\mu(y)\int_{K}g(y^{-1}kxk^{-1})dk\\
&=&\mu(y)[\int_{K}g(y^{-1}kxk^{-1})dk+\mu(x)\int_{K}g(y^{-1}kx^{-1}k^{-1})dk].
\end{eqnarray*}
This implies that
$$\int_{K}g(y^{-1}kxk^{-1})dk+\mu(x)\int_{K}g(y^{-1}kx^{-1}k^{-1})dk
=2g(x)\mu(y^{-1})g(y)=2g(y^{-1})g(x).$$ This completes the proof in the first possibility.\\
 Now we fix $g$ and we consider $$W_{g}=\{ f\in \mathcal{C}(G) : f\; \text{is}\; K-\text{central},\; \text{satisfies}\; (1.1)\; \text{and}\; f(e)=0\}.$$ If $W_{g}\neq \{0\}$, then by using the above computations we get the desired result,
   so we may assume $W_{g}=\{0\}$. Let $f\in \mathcal{C}(G)\setminus\{0\}$ be a solution of (1.1) such that $C_{K}f\neq 0$.
   According to Proposition 2.1 it follows that $C_{K}f$ is a solution of (1.1). Since $W_{g}=\{0\}$
   and $C_{K}(C_{K}f)=C_{K}f$ then $C_{K}f(e)=f(e)\neq 0$. Replacing $C_{K}f$ by $\frac{C_{K}f}{C_{K}f(e)}$, we may assume that $C_{K}f(e)=1$. Let $h$ be a solution of (1.1), then
$C_{K}h-(C_{K}h)(e)C_{K}f\in W_{g}=\{0\}$. So
$C_{K}h=(C_{K}h)(e)C_{K}f$. According to Proposition 2.1 we have for
any $x\in G$ that
$C_{K}(L_{x^{-1}}C_{K}f)(y)=\int_{K}f(xkyk^{-1})dk$ for any $y\in G$
is a solution of (1.1) and that
$C_{K}(C_{K}(L_{x^{-1}}C_{K}f))=C_{K}(L_{x^{-1}}C_{K}f)$. So that
$C_{K}(L_{x^{-1}}C_{K}f)=C_{K}(L_{x^{-1}}C_{K}f)(e)C_{K}f=C_{K}f(x)C_{K}f$.
Then
$\int_{K}C_{K}f(xkyk^{-1})dk=C_{K}(L_{x^{-1}}C_{K}f)(y)=C_{K}f(x)C_{K}f(y)$,
which show that $C_{K}f$ is a $K$-spherical function i.e.
$\int_{K}C_{K}f(xkyk^{-1})dk=C_{K}f(x)C_{K}f(y)$ for all $x,y \in
G$. Substituting this result into
$$\int_{K}C_{K}f(xkyk^{-1})dk+\mu(y)\int_{K}C_{K}f(xky^{-1}k^{-1})dk=2C_{K}f(x)g(y),\;
x, y\in G$$ we get $g(y)=\frac{C_{K}f(y)+\mu(y)C_{K}f(y^{-1})}{2}$.
According to  Proposition 2.1 ii)  we get that $g$ satisfies
equation (2.1). This finishes  the proof of theorem.
\end{proof}
\section{ Study of integral Wilson's functional  equation $W_{\mu}(K)$ on a central pair}
Let $f: G\longrightarrow \mathbb{C}$. For $x\in G$ we define
\begin{equation}\label{eq15} f_{x}(y)=\int_{K}f(xkyk^{-1})dk-f(x)f(y), \; y\in G.\end{equation}
When $f$ is $K$-spherical, then $f_{x}\equiv 0$. A generalized
symmetrized sine addition law is given by
\begin{equation}\label{eq16}
\int_{K}\omega(xkyk^{-1})dk+\int_{K}\omega(ykxk^{-1})dk=2\omega(x)f(y)+2\omega(y)f(x),
\; x,y \in G
\end{equation}
For later use we need the following results :
\begin{prop}([7])
Let $(G,K)$ be a central pair and let $f\in \mathcal{C}(G)$. Then we have\\
i) $f$ satisfies the Kannappan type condition ($*$).\\
ii) If $f$ is $K$-central, then
$$\int_{K}f(xkyk^{-1})dk=\int_{K}f(ykxk^{-1})dk, \; x, y\in G.$$
\end{prop}
As an immediate consequence we get the following corollary
\begin{cor} Let $(G,K)$ be a central pair and let $\omega\in \mathcal{C}(G)$.
If $\omega$ is $K$-central then (3.2) reduces to the generalized
sine addition formula
\begin{equation}\label{eq17}
\int_{K}\omega(xkyk^{-1})dk=\omega(x)f(y)+\omega(y)f(x), \; x,y \in
G.
\end{equation}
\end{cor}
\begin{prop}
Let $f\in \mathcal{C}\setminus\{0\}$ be a solution of the functional equation (1.2). Then\\
i) $f(e)=1$,\\
ii)  $f$ is $K$-central, \\
iii) $f(x)=\mu(x)f(x^{-1})$ for all $x\in G$,\\
iv) $$\int_{K}f(xkyk^{-1})dk=\int_{K}f(ykxk^{-1})dk, \; x, y\in G.$$
v) For any $x\in G$, $(f_{x},f)$ is a solution of (3.2).
\end{prop}
\begin{proof}
i) By setting $y=e$ in (1.2) and by using the fact that $f\neq 0$ we get that $f(e)=1$.\\
ii) For any $x,y\in G$ we have
\begin{eqnarray*}
2f(x)\int_{K}f(kyk^{-1})dk&=&\int_{K}2f(x)f(kyk^{-1})dk\\
&=&\int_{K}(\int_{K}f(xhkyk^{-1}h^{-1})dh\\&+&\mu(y)\int_{K}f(xhky^{-1}k^{-1}h^{-1})dhdk\\
&=&\int_{K}f(xhyh^{-1})dh+\mu(y)\int_{K}f(xhy^{-1}h^{-1})dh\\
&=&2f(x)f(y).
\end{eqnarray*}
So that $f(y)=\int_{K}f(kyk^{-1})dk$ for all $y\in G$, from which we get that $f$ is $K$-central.\\
iii) Since $f$ is $K$-central we get by putting $x=e$ in (1.2) that $f(y)+\mu(y)f(y^{-1})=2f(y)$ for all $y\in G$. Hence $f(y)=\mu(y)f(y^{-1})$ for all $y\in G$.\\
iv) In view of iii) we have for all $x,y\in G$ that
\begin{eqnarray*}
&&\int_{K}f(xkyk^{-1})dk+\mu(y)\int_{K}f(xky^{-1}k^{-1})dk\\&=&2f(x)f(y)\\
&=&2f(y)f(x)\\
&=& \int_{K}f(ykxk^{-1})dk+\mu(x)\int_{K}f(ykx^{-1}k^{-1})dk\\
&=&\int_{K}f(ykxk^{-1})dk+\mu(y)\int_{K}f(xky^{-1}k^{-1})dk.
\end{eqnarray*}
So that we get $\int_{K}f(ykxk^{-1})dk=\int_{K}f(xkyk^{-1})dk$ for all $x, y \in G$.\\
v) In the next we adapt the method used in [35].  According to (1.2)
we get for any $x, y, z$ that
\begin{eqnarray}&&\int_{K}\int_{K}f(xkyk^{-1}hzh^{-1})dkdh+\end{eqnarray}
$$\mu(z)\int_{K}\int_{K}f(xkyk^{-1}hz^{-1}h^{-1})dkdh
=2\int_{K}f(xkyk^{-1})dkf(z),$$
\begin{eqnarray}&&\int_{K}\int_{K}f(xhykz^{-1}k^{-1}h^{-1})dkdh+\end{eqnarray}
$$\mu(yz^{-1})\int_{K}\int_{K}f(xhkzk^{-1}y^{-1}h^{-1})dkdh=
2f(x)\int_{K}f(ykz^{-1}k^{-1})dk,$$
\begin{eqnarray}&&\int_{K}\int_{K}f(xkzk^{-1}hy^{-1}h^{-1})dkdh+\end{eqnarray}
$$\mu(y^{-1})\int_{K}\int_{K}f(xkzk^{-1}hyh^{-1})dkdh=
2\int_{K}f(xkzk^{-1})dkf(y^{-1})$$ from which we get by multiplying
(3.5) by $\mu(z)$ and (3.6) by $\mu(y)$
$$\int_{K}\int_{K}f(xkyk^{-1}hzh^{-1})dkdh+
\mu(z)\int_{K}\int_{K}f(xkyk^{-1}hz^{-1}h^{-1})dkdh$$$$=
2\int_{K}f(xkyk^{-1})dkf(z),$$
$$\mu(z)\int_{K}\int_{K}f(xhykz^{-1}k^{-1}h^{-1})dkdh+
\mu(y)\int_{K}\int_{K}f(xhkzk^{-1}y^{-1}h^{-1})dkdh$$$$=
2\mu(z)f(x)\int_{K}f(ykz^{-1}k^{-1})dk,$$
$$\mu(y)\int_{K}\int_{K}f(xkzk^{-1}hy^{-1}h^{-1})dkdh+
\int_{K}\int_{K}f(xkzk^{-1}hyh^{-1})dkdh$$$$=
2\mu(y)\int_{K}f(xkzk^{-1})dkf(y^{-1}).$$ By subtracting the middle
one from the sum of the two others we get
$$\int_{K}\int_{K}f(xkyk^{-1}hzh^{-1})dkdh+\int_{K}\int_{K}f(xkzk^{-1}hyh^{-1})
dkdh$$
$$=2\int_{K}f(xkyk^{-1})dkf(z)+2\mu(y)\int_{K}f(xkzk^{-1})dkf(y^{-1})$$
$$-2\mu(z)f(x)\int_{K}f(ykz^{-1}k^{-1})dk.$$
Using the fact that $f(y)=\mu(y)f(y^{-1})$ for any $y\in G$ we get
\begin{eqnarray*}
&&\int_{K}\int_{K}f(xkyk^{-1}hzh^{-1})dhdk-f(x)\int_{K}f(ykzk^{-1})dk\\
&+&\int_{K}\int_{K}f(xkzk^{-1}hyh^{-1})dhdk-f(x)\int_{K}f(zkyk^{-1})dk\\
&=&2\int_{K}f(xkyk^{-1})dkf(z)-2f(x)\int_{K}f(yhzk^{-1})dh\\
&+&2\int_{K}f(xkzk^{-1})dkf(y)-2f(x)[2f(y)f(z)-\int_{K}f(yhzk^{-1})dh]\\
&=&2\int_{K}f(xkyk^{-1})dkf(z)-2f(x)f(y)f(z)\\
&+&2\int_{K}f(xkzk^{-1})dkf(y)-2f(x)f(y)f(z).\end{eqnarray*} So we
have for any $x,y, z\in G$ that
$$\int_{K}f_{x}(ykzk^{-1})dk+\int_{K}f_{x}(zkyk^{-1})dk
=2f_{x}(y)f(z)+ 2f_{x}(z)f(y).$$ Hence for all $x\in G$, the pair
$(f_{x},f)$ is a solution of (3.2).
\end{proof}
\begin{thm}
Let $f, g\in \mathcal{C}(G)$ be a solution of the functional
equation
\begin{equation}\label{eq17}
\int_{K}f(xkyk^{-1})dk=f(x)g(y)+g(x)f(y), \; x,y \in G.
\end{equation} Then one of the following statements hold \\
i) $f= 0$ and $g$ arbitrary in $\mathcal{C}(G)$.\\
ii) There exists a $K$-spherical function $\varphi$  and a non-zero
constant $c\in \mathbb{C}^{*}$ such that
$$g=\frac{\varphi}{2}, \; f=c\varphi.$$
iii) There exist two $K$-spherical functions $\varphi$, $\psi$ for
which $\varphi \neq \psi$ and a non-zero constant $c\in
\mathbb{C}^{*}$ such that $$g=\frac{\varphi+\psi}{2}, \;
f=c(\varphi-\psi).$$ iv) $g$ is a $K$-spherical function and $f$ is
associated to $g$.
\end{thm}
\begin{proof}
 The proof is similar to one used in  [20].
\end{proof}
\section{Solution of equation $W_{\mu}(K)$ on a central pair}
In this section we obtain solution of equation $(1.1)$ in the case
where $(G,K)$ is a central pair.\\\\ In the next theorem we solve
the functional equation $(1.2)$. We will adapt the method used in
[35].
\begin{thm}
Let $(G,K)$ be a central pair. Let $\mu: G\longrightarrow
\mathbb{C}^{*}$ be a character on $G$ and let $f: G\longrightarrow
\mathbb{C}$ be a non-zero solution of the functional equation (2.1).
 Then there exists a $K$-spherical function $\varphi : G\longrightarrow \mathbb{C}$ such that
$f=\frac{\varphi+\mu\check{\varphi}}{2}.$
\end{thm}
\begin{proof}
By using  5i) in Proposition 3.3 we get for all $x\in G$ that the pair $(f_{x},f)$ is a solution of (3.2).\\
First case : There exists $x\in G$ such that $f_{x} \neq 0$.
According to iii) in Theorem 3.4, there exist two $K$-spherical
functions $\varphi$ and $\psi$ such that $\varphi \neq\psi$ and that
$f=\frac{\varphi+\psi}{2}$. By substituting this in (1.2) we get for
all $x,y \in G$ that
$$\varphi(x)[\mu(y)\varphi(y^{-1})-\psi(y)]+\psi(x)[\mu(y)\psi(y^{-1})-\varphi(y)]=0.$$
 Since $\varphi\neq \psi$, according to [19] we get that $\varphi$ and $\psi$ are linearly independent. Hence $\psi(y)=\mu(y)\varphi(y^{-1})$ for any $y\in G$.\\
By iv) of Theorem 3.4  we get by a small computation the desired result.\\
Second case : if $f_{x}=0$, for all $x\in G$  then
$\int_{K}f(xkyk^{-1})d\omega_{K}(k)=f(x)f(y)$ for all $x, y\in G$.
Then by subsisting $f$ in (1.2) we get that $f(x)=\mu(x)f(x^{-1})$
for all $x\in G$.  Hence
$f=\varphi=\frac{\varphi+\mu\check{\varphi}}{2}$ where $\varphi$ is
$K$-spherical function.
\end{proof}
In the next theorem we solve the functional equation (1.1).
\begin{thm}
Let $(G,K)$ be a central pair. If $(f,g)$ is a solution of (1.1)
such that  $C_{K}f\neq 0$,
then there exists a $K$-spherical function $\varphi$ such that\\
1) $$g=\frac{\varphi+\mu\check{\varphi}}{2}.$$
2) \\
i) When $\varphi \neq \mu\check{\varphi}$, then there exist $\alpha,
\beta \in \mathbb{C}$ such that
$$f=\alpha\frac{\varphi+\mu\check{\varphi}}{2}+\beta\frac{\varphi-\mu\check{\varphi}}{2}.$$
ii) When $\varphi =\mu\check{\varphi}$, then $f=\alpha\varphi+l$
where $\alpha \in \mathbb{C}$ and $l$ is a solution of the
functional equation
\begin{equation}
\int_{K}l(xkyk^{-1})dk=l(x)\varphi(y)+\varphi(x)l(y), \; x,y\in G.
\end{equation}
\end{thm}
\begin{proof}
Let $f, g\in \mathcal{C}(G) \setminus\{0\}$ be a solution of (1.2),
then by Theorem 2.3 we get that $g$ is a solution of (1.2).
According to Theorem 4.1 there exists a $K$-spherical function such
that  $g(x)=\frac{\varphi(x)+\mu(x)\varphi(x^{-1})}{2}$ for any
$x\in G$. By decomposing $f$ in the following way we get for any
$x\in G$ that
\begin{eqnarray*}f(x)&=&\frac{f(x)+\mu(x)f(x^{-1})}{2}+\frac{f(x)-\mu(x)f(x^{-1})}{2} \\&=&
f_{1}(x)+f_{2}(x)\end{eqnarray*} where
$f_{1}(x)=\frac{f(x)+\mu(x)f(x^{-1})}{2}$ and
$f_{2}(x)=\frac{f(x)-\mu(x)f(x^{-1})}{2}$ for all $x\in G$. By easy
computations we get that $f_{1}(x)=\mu(x)f_{1}(x^{-1})$ and
$f_{2}(x)=-\mu(x)f_{2}(x^{-1})$ for all $x\in G$. By using the fact
that $\int_{K}f(xkyk^{-1})dk=\int_{K}f(ykxk^{-1})dk$ for all $x,
y\in G$ we get that
\begin{equation}
 \int_{K}f_{1}(xkyk^{-1})dk+\mu(y)\int_{K}f_{1}(xky^{-1}k^{-1})dk=2f_{1}(x)g(y),\;\;x,y\in G.
\end{equation}
By setting $x=e$ in (4.3) we get that $f_{1}(y)=f_{1}(e)g(y)=\alpha
g(y)$ for any $y\in G$. On the other hand by  small computations we
show that $f_{2}$ is a solution of the functional equation
\begin{equation}
 \int_{K}f_{2}(xkyk^{-1})dk=f_{2}(x)g(y)+f_{2}(y)g(x),\;\;x,y\in G.
\end{equation}
According to iii) and iv) in Theorem 3.4  we get the remainder.
\end{proof}
In the next corollary we use R. Godement's spherical functions theory [27] to give explicit formulas in terms of irreducible representations of $G$ :  Let $(\pi,\mathcal{H})$ be a completely irreducible representation of $G$, $\delta$ be an irreducible representation of $K$ and $\chi_{\delta}$ the normalized character of $\delta$. The set of vectors in $\mathcal{H}$ which under $k\longrightarrow \pi(k)$ transform according to $\delta$ is denoted by $\mathcal{H}_{\delta}$. The operator $E(\delta)=\int_{K}\pi(k)\overline{\chi_{\delta}}(k)dk$ is a continuous projection of $\mathcal{H}$ onto $\mathcal{H}_{\delta}$.\\\\
We will say that a function $f$ is quasi-bounded if there exists a
semi-norm $\rho(x)$ such that $\sup_{x\in
G}\frac{|f(x)|}{\rho(x)}<+\infty$ (see [27]). According to [7,
Theorem 4.1, Theorem 6.2] we have the following corollary
\begin{cor}
Let $(G,K)$ be a central pair. Let $f$ be a non-zero quasi-bounded
continuous function on $G$ satisfying $\chi_{\delta}\star f=f$. Then
$f$ is a solution of (1.2) if and only if there exists a completely
irreducible representation $(\pi,\mathcal{H})$ of $G$ such that
$$f(x)=\frac{1}{2dim(\delta)}(tr(E(\delta)\pi(x))+\mu(x)tr(E(\delta)\pi(x^{-1})),\; x\in G$$
where $tr$ is the trace on $\mathcal{H}$ and $dim(\delta)$ is the
dimension of $\delta$.
\end{cor}
In the next corollary we assume that $G$ is a compact. Then $(G,G)$
is a central pair (see [7] and [9]).
\begin{cor}
Let $G$ be a compact group and let $f$ be a continuous function on
$G$. Then $f$ is is a solution of (1.8) if and only if there exists
a continuous irreducible representation $(\pi,\mathcal{H})$ such
that
$$f(x)=\frac{\chi_{\pi}(x)+\mu(x)\chi_{\pi}(x^{-1})}{2d(\pi)}, \; x\in G$$
where $\chi_{\pi}$ and $d(\pi)$ are respectively the character and
the dimension of $\pi$.
\end{cor}
\section{Superstablity  of $W_{\mu}(K)$ on a locally compact group}
In this section, by using Theorem 2.4, we study the superstability
problem of equations $W_{\mu}(K)$ and  $D_{\mu}(K)$ on non abelian
case
\begin{lem}
Let $\delta>0$. Let $\mu: G\longrightarrow \mathbb{C}$ be a unitary
character of $G$. let  $f, g \in \mathcal{C}(G)$ such that $f$ is
unbounded and $(f, g)$ is a solution of the inequality \small
\begin{equation}
 |\int_{K}f(xkyk^{-1})d\omega_{K}(k)+\mu(y)\int_{K}f(xky^{-1}k^{-1})dk-2f(x)g(y)|\leq \delta,\;\;x,y\in G.
\end{equation} Then\\
i) For all $x\in G$, $(C_{K}(L_{x^{-1}}f),g)$ is a solution of the inequality (5.1).\\
ii) $g(y)=\mu(y)g(y^{-1})$ for all $y\in G$.\\
iii) $g$ is $K$ central.
\end{lem}
\begin{proof}
i) and iii) by easy computation.\\
ii) Since $(f, g)$ is a solution of (5.1), then we get for any $x,
y\in G$
$$|\int_{K}f(xkyk^{-1})dk+\mu(y)\int_{K}f(xky^{-1}k^{-1})dk-2f(x)g(y)|\leq \delta$$
and
$$|\int_{K}f(xky^{-1}k^{-1})dk+\mu(y^{-1})\int_{K}f(xkyk^{-1})dk-2f(x)g(y^{-1})|\leq \delta.$$
By multiplying the last inequality by $\mu(y)$ we get that
$$|\mu(y)\int_{K}f(xky^{-1}k^{-1})dk+\int_{K}f(xkyk^{-1})dk-2f(x)\mu(y)g(y^{-1})|\leq \delta.$$
By triangle inequality we get
$$|2f(x)||g(y)-\mu(y)g(y^{-1})|\leq 2\delta$$ for all $y\in G$. Since $f$ is unbounded we get that $g(y)=\mu(y)g(y^{-1})$
for all $y\in G$.\end{proof}
\begin{lem}
Let $\delta>0$. Let $\mu: G\longrightarrow \mathbb{C}^{*}$ be a
unitary character of $G$. Let  $f, g \in \mathcal{C}(G)$ such that
$g$ is unbounded solution of inequality (5.1). Then $g$ is a
solution of d'Alembert's functional equation (1.2).
\end{lem}
\begin{proof}
By using the same method as in [14, Corollary 2.7 iii] we get that
$g$  is a solution of (1.2)
\end{proof}
\begin{lem}
Let $\delta>0$. Let $\mu: G\longrightarrow \mathbb{C}^{*}$ be a
unitary character of $G$. Let  $f, g \in \mathcal{C}(G)$ such that
$f$ is unbounded solution of inequality (5.1) and that the function
$x\longrightarrow f(x)+\mu(x)f(x^{-1})$ is bounded and $C_{K}f\neq
0$. Then $g$ is a solution of d'Alembert's functional equation
(1.2).
\end{lem}
\begin{proof}
First case: $g$ is bounded. Let
$$\psi(x,y)=\int_{K}g(xkyk^{-1})dk+\mu(y)\int_{K}g(y^{-1}kxk^{-1})dk-2g(x)g(y),\;
x, y\in G.$$
 Then according to the proof of  Theorem 2.4 we get for all $x, y, z\in G$ that
$$2f(z)\psi(x,y)+2f(y)\psi(x,z)=2g(x)[\int_{K}f(zkyk^{-1})dk+\int_{K}f(ykz^{-1}k^{-1})dk]$$
$$-\mu(z)\int_{K}\int_{K}f(yhxkz^{-1}k^{-1}h^{-1})dkdh-\mu(x)\mu(y)
\int_{K}\int_{K}f(zhx^{-1}ky^{-1}k^{-1}h^{-1})dkdh$$$$
-\mu(y)\int_{K}\int_{K}f(zhxky^{-1}k^{-1}h^{-1})dkdh-\mu(x)\mu(z)
\int_{K}\int_{K}f(yhx^{-1}kz^{-1}k^{-1}h^{-1})dkdh.$$ Since the
function $x\longrightarrow f(x)+\mu(x)f(x^{-1})$ is bounded, then
there exists $\beta>0$ such that $|f(x)+\mu(x)f(x^{-1})|\leq \beta$
for all $x\in G$. So that we get for any $x, y, z\in G$ that
\begin{equation}|2f(z)\psi(x,y)+2f(y)\psi(x,z)|\leq 2\beta(1+|g(x)|).\end{equation}
Since $f$ is unbounded, then there exists a sequence $(z_{n})_{n\in \mathbb{N}}$ such that $\lim_{n\longrightarrow \infty}|f(z_{n})|=+\infty$. By using the inequality (5.2), it follows that there exists $c_{x}\in \mathbb{C}$ such that $\psi(x,y)=c_{x}f(y)$ for all $x, y\in G$. So that the function $(x,y,z)\longrightarrow 2f(z)c_{x}f(y)+2f(y)c_{x}f(z)$ is bounded. Since $f$ is unbounded it follows that $c_{x}=0$ for all $x\in G$. As in the proof of Theorem 2.4, we get that $g$ is a solution of functional equation (1.2).\\
Second case: $g$ is unbounded. According to lemma 5.2 we get that
$g$ is a solution of (1.2).
\end{proof}
\begin{lem}
Let $\delta>0$. Let $\mu: G\longrightarrow \mathbb{C}^{*}$ be a
unitary character of $G$. Let  $f, g \in \mathcal{C}(G)$ such that
$f$ is an unbounded solution of inequality (5.1).
 Then
$g$ is a solution of d'Alembert's long functional equation
\begin{equation}
\int_{K}g(xkyk^{-1})dk+\mu(y)\int_{K}g(xky^{-1}k^{-1})dk+\int_{K}g(ykxk^{-1})dk+\end{equation}
$$\mu(y)\int_{K}g(y^{-1}kxk^{-1})dk=4g(x)g(y)$$
for all $x, y\in G$.
\end{lem}
\begin{proof} For all $x, y, z\in G$ we have
$$2|f(z)||\int_{K}g(xkyk^{-1})dk+\mu(y)\int_{K}g(xky^{-1}k^{-1})dk+\int_{K}g(ykxk^{-1})dk
$$$$+
\mu(y)\int_{K}g(y^{-1}kxk^{-1})dk-4g(x)g(y)|$$
$$\leq |\int_{K}\int_{K}f(zhxkyk^{-1}h^{-1})dkdh+
\mu(xy)\int_{K}\int_{K}f(zhy^{-1}kx^{-1}k^{-1}h^{-1})dkdh $$$$
-2f(z)\int_{K}g(xkyk^{-1})dk|$$
$$+|\mu(y)\int_{K}\int_{K}f(zhx^{-1}kyk^{-1}h^{-1})dkdh+\mu(x)
\int_{K}\int_{K}f(zhykx^{-1}k^{-1}h^{-1})dkdh$$$$
-2\mu(y)f(z)\int_{K}g(xky^{-1}k^{-1})dk|$$$$+|\int_{K}\int_{K}f(zhykxk^{-1}h^{-1})dkdh+
\mu(yx)\int_{K}\int_{K}f(zhx^{-1}ky^{-1}k^{-1}h^{-1})dkdh $$$$
-2f(z)\int_{K}g(ykxk^{-1})dk|$$
$$+|\mu(y)\int_{K}\int_{K}f(zhy^{-1}kxk^{-1}h^{-1})dkdh+\mu(x)
\int_{K}\int_{K}f(zhx^{-1}kyk^{-1}h^{-1})dkdh$$$$
-2\mu(y)f(z)\int_{K}g(y^{-1}kxk^{-1})dk|$$
$$+\int_{K}\int_{K}f(zhxky^{-1}k^{-1}h^{-1})dkdh+
\mu(y)\int_{K}\int_{K}f(zhxky^{-1}k^{-1}h^{-1})dkdh $$$$
-2\int_{K}f(zkxk^{-1})dkg(y)|$$
$$+\int_{K}\int_{K}f(zhykxk^{-1}h^{-1})dkdh+
\mu(x)\int_{K}\int_{K}f(zhyh^{-1}kx^{-1}k^{-1})dkdh $$$$
-2\int_{K}f(zkyk^{-1})dkg(x)|$$
$$+|\mu(y)\int_{K}\int_{K}f(zhy^{-1}h^{-1}kxk^{-1})dkdh+\mu(y)\mu(x)
\int_{K}\int_{K}f(zhy^{-1}h^{-1}kx^{-1}k^{-1}h^{-1})dkdh$$
$$-2\mu(y)\int_{K}f(zky^{-1}k^{-1})dkg(x)|$$$$
+|\mu(x)\int_{K}\int_{K}f(zhx^{-1}h^{-1}kyk^{-1})dkdh+\mu(y)\mu(x)
\int_{K}\int_{K}f(zhx^{-1}h^{-1}ky^{-1}k^{-1}h^{-1})dkdh$$
$$-2\mu(x)\int_{K}f(zkx^{-1}k^{-1})dkg(y)|$$$$
+2|g(y)||\int_{K}f(zkxk^{-1})dk+\mu(x)\int_{K}f(zkx^{-1}k^{-1})dk-2f(z)g(x)|$$
$$+2|g(x)||\int_{K}f(zkyk^{-1})dk+\mu(y)\int_{K}f(zky^{-1}k^{-1})dk-2f(z)g(y)|$$
$$\leq (8+|g(y)|+|g(x))\delta.$$
Since $f$ is bounded we conclude that $g$ is solution of functional
equation (5.3).
\end{proof}
In the next corollary we assume that $K$ is a discrete subgroup of
$G$
\begin{lem}
Let $\delta>0$. Let $\mu: G\longrightarrow \mathbb{C}$ be a unitary
character of $G$. Let  $f, g \in \mathcal{C}(G)$ such that $f$ is an
unbounded $K$-central solution of the inequality (5.1) such that
$C_{K}f\neq 0$.
 Then
$g$ is a solution of d'Alembert's functional equation (1.2).
\end{lem}
\begin{proof}
Since $(f, g)$ is a solution of the inequality (5.1) we get
according to lemma 5.1 that $(C_{K}f,g)$ is also a solution of
(5.1). By setting $x=e$ in (5.1) and by the fact that $C_{K}f$ is
$K$-central it follows that
$$|C_{K}f(y)+\mu(y)C_{K}f(y^{-1})-2f(e)g(y)|\leq \delta, \; y\in G.$$
If $f(e)=0$ we get that the function $y\longrightarrow C_{K}f(y)+\mu(y)C_{K}f(y^{-1})$ is bounded. According to Lemma 5.3 we have that $g$ is a solution of d'Alembert's functional equation (1.2).\\
If $f(e)\neq 0$. Replacing $C_{K}f$ by $\frac{C_{K}f}{f(e)}$ we may assume that $f(e)=1$. Consider the function $C_{K}(L_{a}f)(x)=\int_{K}f(akxk^{-1})d\omega_K(k)$ for all $x\in G$. According to lemma 5.1 we get that $(C_{K}(L_{a}f),g)$  is a solution of (5.1). Let $a\in G$ such that $$h=C_{K}(L_{a}f)-C_{K}(L_{a}f)(e)C_{K}f.$$ If there exists $a\in G$ such that $h$ is unbounded on $G$. Since $h(e)=0$ and $C_{K}h=h$ and that the function $h$ is a solution of the inequality (5.1) it follows that $x\longrightarrow h(x)+\mu(x)h(x^{-1})$ is bounded. According to lemma 5.3 we get that $g$ is a solution d'Alembert's functional equation (1.2).\\
Now assume that $h$ is bounded, that is there exists $M(x)>0$ such
that $$|\int_{K}f(xkyk^{-1})dk-f(x)C_{K}f(y)|\leq M(x)$$ for all
$x,y\in G$. Since $f$ is $K$-central we get for all $x, y\in G$ that
$$|\int_{K}f(xkyk^{-1})dk-f(x)f(y)|\leq M(x).$$ By using triangle inequality we get for all $x, y, z\in G$ that
$$|f(z)||\int_{K}f(xkyk^{-1})dk-f(x)f(y)|$$
$$\leq |-\int_{K}\int_{K}f(xkyk^{-1}hzh^{-1})kdh+\int_{K}f(xkyk^{-1})dkf(z)|$$
$$+|\int_{K}\int_{K}f(xkyk^{-1}hzh^{-1})dkdh-f(x)\int_{K}f(ykzk^{-1})dk|$$
$$+|f(x)||\int_{K}f(ykzk^{-1})dk-f(y)f(z)|$$
$$\leq \int_{K}M(xkyk^{-1})dk+M(x)+|f(x)|M(y)|.$$ Since $f$ is unbounded it follows that $\int_{K}f(xkyk^{-1})dk-f(x)f(y)$ for all $x, y\in G$. Substituting this result into inequality (5.1) it follows that $$|f(x)||f(y)+\mu(y)f(y^{-1})-2g(y)|\leq \delta$$ for all $x,y \in G$. Since $f$ is unbounded we get that $g(y)=\frac{f(y)+\mu(y)f(y^{-1})}{2}$ for all $y\in G$. According to Proposition 2.1 we get that $g$ is a solution of d'Alembert's functional equation (1.2).
\end{proof}
The next theorem is the main result of this section
\begin{thm} Let $\delta>0$ be fixed, $\mu$ be a unitary character of $G$ and let  $f, g :G\longrightarrow \mathbb{C}$ such that $(f,g)$ satisfies (5.1) and $f$ is $K$-central.
 Then \\ 1) $f, g$ are bounded or\\
2) $f$ is unbounded and $g$ satisfies d'Alembert's  functional
equation (1.2) or\\ 3) $g$  is unbounded and f satisfies the
functional equation (1.1) (if $f\neq 0$  such that $C_{K}f\neq 0$,
then g satisfies the d'Alembert's  functional equation (1.2)).
\end{thm}
\begin{proof}
We get 1) by easy computations.\\
2) Assume that $f$ is unbounded.  According to Lemmas 5.1, 5.3 and 5.5 we get the proof.\\
3) Assume that $g$ is unbounded, then for $f=0$ the pair $(f,g)$ is
a solution of equation (5.1). Afterward we suppose that $f\neq 0$.
By using (5.1) and the following decomposition \small
\begin{eqnarray*}
&&2|g(z)||\int_{K}f(xkyk^{-1})dk+\mu(y)\int_{K}f(xky^{-1}k^{-1})dk-2f(x)g(y)|\\&=&
|-2g(z)\int_{K}f(xkyk^{-1})dk-2g(z)\mu(y)\int_{K}f(xky^{-1}k^{-1})dk-4g(z)f(x)g(y)|\\
&\leq &|\int_{K}\int_{K}f(xkyk^{-1}hzh^{-1})dkdh+\mu(z)
\int_{K}\int_{K}f(xkyk^{-1}hz^{-1}h^{-1})dkdh|\\&-&2\int_{K}f(xkyk^{-1})g(z)dk
\\&+& |\mu(y)\int_{K}\int_{K}f(xky^{-1}k^{-1}hzh^{-1})dkdh+
\mu(y)\mu(z)\int_{K}\int_{K}f(xky^{-1}k^{-1}hz^{-1}h^{-1})dkdh\\&-&2\mu(y)
\int_{K}f(xky^{-1}k^{-1})g(z)dk|\\
&+&|\int_{K}\int_{K}f(xhykzk^{-1}h^{-1})dkdh
+\mu(yz)\int_{K}\int_{K}f(xhkz^{-1}k^{-1}y^{-1}h^{-1})dkdh\\
&-&2f(x)\int_{K}g(ykzk^{-1})dk|\\&+&
|\mu(z)\int_{K}\int_{K}f(xhkyk^{-1}z^{-1}h^{-1})dkdh+
\mu(z)\mu(yz^{-1})\int_{K}\int_{K}f(xkzk^{-1}hy^{-1}h^{-1})dkdh\\&-&2\mu(z)
\int_{K}f(x)g(ykz^{-1}k^{-1})dk|\\&+&
|\mu(z)\int_{K}\int_{K}f(xhz^{-1}kyk^{-1}h^{-1})dkdh+
\mu(z)\mu(z^{-1}y)\int_{K}\int_{K}f(xhky^{-1}k^{-1}zh^{-1})dkdh\\&-&2\mu(z)
\int_{K}f(x)g(z^{-1}kyk^{-1})dk|\\&+&
|\int_{K}\int_{K}f(xhzkyk^{-1}h^{-1})dkdh+
\mu(zy)\int_{K}\int_{K}f(xhky^{-1}k^{-1}z^{-1}h^{-1})dkdh
\\&-&2f(x)\int_{K}g(zkyk^{-1})dk|\\&+&
|\mu(z)\int_{K}\int_{K}f(xhz^{-1}kyk^{-1}h^{-1})dkdh+
\mu(z)\mu(y)\int_{K}\int_{K}f(xkz^{-1}k^{-1}hy^{-1}h^{-1})dkdh\\&-&2\mu(z)
\int_{K}f(xkz^{-1}k^{-1})g(y)dk|\\&+&
|\int_{K}\int_{K}f(xkzk^{-1}hyh^{-1})dkdh+
\mu(y)\int_{K}\int_{K}f(xkzk^{-1}hy^{-1}h^{-1})dkdh
\\&-&2\int_{K}f(xkzk^{-1})g(y)dk|\\&+&
2|f(x)||\int_{K}g(ykzk^{-1})dk+\mu(z)\int_{K}g(ykz^{-1}k^{-1})dk
+\int_{K}g(zkyk^{-1})dk\\&+&\mu(z)\int_{K}g(z^{-1}kyk^{-1})dk-4g(y)g(z)|\\&+&
2|g(y)||\int_{K}f(xkzk^{-1})dk+\mu(z)\int_{K}f(xkz^{-1}k^{-1})dk-2f(x)g(z)|\\&&\leq
\delta+|\mu(y)|\delta+\delta+2|\mu(z)|\delta+\delta+|\mu(z)|\delta+\delta+2|f(x)|\times 0+2|g(y)|\delta\\
&=&\delta(8+2|g(y)|).
\end{eqnarray*}
Since $g$ is unbounded it follows that $f,g$ satisfy the functional
equation (1.1). According to Theorem 2.4 we get the remainder
\end{proof}
As a consequence we get the superstability of the functional
equations (1.2), (1.3), (1.4) and (1.7)
\begin{cor}
Let $\delta>0$ be fixed, $\mu$ be a unitary character of $G$. Let
$f: G\longrightarrow \mathbb{C}$ such that
\begin{equation}\label{eq31}
 |\int_{K}f(xkyk^{-1})dk+\mu(y)\int_{K}f(xky^{-1}k^{-1})dk-2f(x)f(y)|\leq \delta,\;\;x,y\in G.
\end{equation}
Then either $f$ is bounded or $f$ is a solution of the functional
equation (1.2).
\end{cor}
Let $f(kxh)=\chi(k)f(x)\chi(h)$, $k, h\in K$ and $x\in G$. Then we
have the following corollary
\begin{cor}
Let $\delta>0$ be fixed, $\mu$ be a bounded character of $G$. Let
$f: G\longrightarrow \mathbb{C}$ such that
\begin{equation}\label{eq31}
 |\int_{K}f(xky)\overline{\chi(k)}dk+\mu(y)\int_{K}f(xky^{-1})\overline{\chi(k)}dk-2f(x)f(y)|\leq \delta,\;\;x,y\in G.
\end{equation}
Then either $f$ is bounded or $f$ is a solution of the functional
equation (1.7).
\end{cor}
In the next corollary we assume that $K\subset Z(G)$. Then we get
\begin{cor}
Let $\delta>0$ be fixed, $\mu$ be a unitary character of $G$. Let $f
: G\longrightarrow \mathbb{C}$ such that
\begin{equation}\label{eq31}
 |f(xy)+\mu(y)f(xy^{-1})-2f(x)f(y)|\leq \delta,\;\;x,y\in G.
\end{equation}
Then either $f$ is bounded or $f$ is a solution of the functional
equation (1.3).
\end{cor}

\vspace{1cm}
Belaid Bouikhalene\\ Departement of Mathematics and Informatics\\
Polydisciplinary Faculty, Sultan Moulay Slimane university, Beni Mellal, Morocco.\\
E-mail : bbouikhalene@yahoo.fr.\\\\
 Elhoucien Elqorachi, \\Department of Mathematics,
\\Faculty of Sciences, Ibn Zohr University, Agadir,
Morocco,\\
E-mail: elqorachi@hotmail.com\\\\

\end{document}